\documentclass[reqno]{amsart}

\usepackage[a4paper,margin=1.5in]{geometry}

\usepackage{mathrsfs}
\usepackage{pdfpages}
\usepackage{multicol}
\usepackage[absolute, overlay]{textpos}

\usepackage{amsfonts,  amsmath,  mathabx}
\usepackage{graphicx}
\usepackage{enumerate,  enumitem,  amscd}
\usepackage{amsthm,  amssymb,  epsfig,  psfrag,  hyperref,  amsmath}  
\usepackage{mathtools}

\usepackage{tikz}
\usepackage{tikz-cd}
\usetikzlibrary{arrows,automata}

\newcommand{\Conv}{\mathop{\scalebox{1.5}{\raisebox{-0.2ex}{$\ast$}}}}%

\newcommand{\pres}[3]{\textnormal{#1} \langle #2 \mid #3 \rangle}

\newcommand{\Z}{\mathbb{Z}}

\newcommand{\Q}{\mathbb{Q}}

\DeclareMathOperator{\SL}{SL}
\DeclareMathOperator{\GL}{GL}

\newtheorem{theorem}{Theorem} % 1st argument is your name for it
\newtheorem*{theorem*}{Theorem} % 1st argument is your name for it
\newtheorem*{HEtheorem}{Higman's Embedding Theorem} 
\numberwithin{theorem}{section}
\newtheorem*{blemma}{Britton's Lemma} % 1st argument is your name for it
\newtheorem{corollary}[theorem]{Corollary}

\numberwithin{lemma}{section}
\numberwithin{proposition}{section}

\newtheorem*{thmI}{Theorem I}
\newtheorem*{thmII}{Theorem II}
\newtheorem*{thmIII}{Theorem III}
\newtheorem*{thmIV}{Theorem IV}

\theoremstyle{definition}

\newtheorem*{question*}{Question}

\numberwithin{example}{section}
\newtheorem{remark}{Remark}
\newtheorem{remark*}{Remark}

\def\<{\langle}
\def\>{\rangle}
\def\-{\overline}

\def\G{\Gamma}



\begin{document}

\title[HNN Extensions and Embedding Theorems for Groups]{HNN Extensions and Embedding Theorems for Groups}

\author{Martin R. Bridson}
\address{Mathematical Institute, Andrew Wiles Building, Oxford OX2 6GG, United Kingdom}
\email{bridson@maths.ox.ac.uk}

\author{Carl-Fredrik Nyberg-Brodda}
\address{June E Huh Center for Mathematical Challenges, Korea Institute for Advanced Study (KIAS), Seoul 02455, Korea}
\email{cfnb@kias.re.kr}

\thanks{In celebration of the centenary of the Journal of the London Mathematical Society.}

\date{\today}


\begin{abstract}
The \textit{Higman-Neumann-Neumann (HNN) paper} of 1949 is a landmark of group theory in the twentieth century. The proof of its main theorem covers less than a page and uses only pre-existing technology, but the construction that it introduced -- the \textit{HNN extension} -- quickly became one of the principal tools of combinatorial group theory,  widely used to  build new groups and to describe enlightening decompositions of existing groups.  In this article,  we shall describe the contents of the HNN paper, and  then discuss some of the important developments that followed in its wake,  leading up to the 
central role that HNN extensions play in the Bass--Serre theory of groups acting on trees.
\end{abstract}

\maketitle

%\begin{textblock*}{20cm}(11cm,0.3cm)
%\fbox{\centering \texttt{Draft of \today, (merged, cut version)}} \\
%\end{textblock*}

\noindent  \textit{HNN extensions} play a central role in geometric and combinatorial group theory.
They provide the principal tool for embedding groups into larger groups with prescribed properties, and they appear in topology when one decomposes the fundamental group of a space by cutting along a subspace. HNN extensions were introduced in the 1949 paper from which they derive their name, ``{\em Embedding Theorem for Groups}" \cite{HNN1949}, by G.\ Higman, B.\ H.\ Neumann, and H.\ Neumann. This remarkable paper is a model of concise mathematical writing: after a brief introduction explaining the contents, the main theorem and two major applications are exposed with clarity in less than six pages. From a modern perspective, the abiding importance of this article lies principally with its introduction of HNN extensions,  but it also contains an abundance of ideas that foreshadow
significant later developments.  

The existence of HNN extensions is the content of Theorem I of \cite{HNN1949} which 
provides a definitive answer to the question posed in the opening sentences of the article: {\em ``A group $G$ contains two subgroups $A$ and $B$. The question then arises whether a group $H$ exists, containing $G$, within which $A$ and $B$ are conjugate.}  Theorem I proves, more precisely, that any specified isomorphism from $A$ to $B$ can be realized as a conjugacy in a group $H$ that contains $G$.  Immediately after the theorem, the authors observe that there is a (canonical) minimal choice for $H$ as they recast Theorem I in the language  of {\em adjunctions} of solutions to systems of equations,  a framework that originates in an earlier paper by the second author. It is this canonical minimal choice that is now called {\em the HNN extension of $G$ associated to an isomorphism $\phi:A\to B$};  
no notation is dedicated to this group in \cite{HNN1949}, but subsequently the notation  $G\ast_\phi$ became fairly standard.  

Following Theorem I,  the authors of \cite{HNN1949} quickly generalize their main construction to cover families of isomorphic subgroups (Theorem II),  a generalization that is needed for the two applications of historic importance that they give. As a first application,  they prove that every
torsion-free group $G$ can be embedded in a group $G^*$ in which each pair of non-trivial elements is conjugate, and thereby provide the first examples of infinite,  torsion-free groups that are both countable and simple.   Secondly,  they prove that every countable group can be embedded in a group that requires only two generators. Throughout, they pay close attention to the economy of their embeddings with respect to properties such as countability,  finite generation, and finite presentation.   

Our purpose in the present article is twofold.  First, following some preliminary remarks and the setting of notation,  we will describe the contents of \cite{HNN1949} in more detail, retaining the original structure of the paper and the notation used in the statement of the theorems. We shall trace out the mathematical context in which this work took place,  outline the original proofs, and  comment briefly on how some of the ideas in the proofs subsequently evolved. In the second half of this article, we describe some of the major developments that flowed in the wake of the HNN construction, dividing these developments into two parts. First, with a focus on developments in the  1950s and 1960s, we discuss how HNN extensions are applied to build new groups. We emphasize in particular the way  they were used to clarify the existence of finitely presented groups with unsolvable word problem, and the crucial role that they played in Higman's celebrated Embedding Theorem \cite{Higman1961}; these applications were firmly in the realm of {\em combinatorial} group theory. Turning to developments beyond the 1960s, we emphasize the  role HNN extensions play in describing decompositions of groups. The flavour of these developments
is more topological, and we begin with a discussion of
van Kampen's theorem,  where  HNN extensions appear as one tries to calculate fundamental groups of spaces. This leads us to a discussion of graphs of groups and the context in which HNN extensions are studied today: the  \textit{Bass--Serre theory} of groups acting on trees, which has become a cornerstone of geometric group theory and low-dimensional topology.

\subsection*{Acknowledgments}

We thank Hyman Bass and Jean-Pierre Serre for their detailed responses to our questions about the early history of Bass--Serre theory.  We are also grateful to Ian Chiswell, Daniel E.\ Cohen, Chris Hollings, John Stillwell, and David Xu for helpful discussions.

\section{Context and Notation}\label{Sec:introduction}

\noindent For the most part,  the groups that Higman, Neumann and Neumann consider in \cite{HNN1949} are described 
using generators and relations, a formalism that was well established by 1949.   There is no whiff of topology or group
actions in \cite{HNN1949}.
From the canon of pre-existing  knowledge, the main tool they used was the theory of amalgamated free products,
and that it what we shall discuss in this section.
We will also settle on the notation to be used in the remainder of this article,   and comment on how 
notation varies in the literature.

\subsection{Amalgamated free products}\label{Subsec:amalgamated-free-products}

Amalgamated free products were introduced by Schreier in 1927 \cite{Schreier1927}, 
in connection with his proof of the \textit{Nielsen--Schreier theorem},  which states
that every subgroup of a free group is free.  They rapidly become a central tool in combinatorial group theory;
for example, \ Magnus \cite{Magnus1932} used them in 1932 in his solution to the word problem for  one-relator groups. 

We remind the reader that if $V_0$ and $V_1$ are groups and $\phi: U_0\to U_1$ is an isomorphism between subgroups $U_0<V_0$ and $U_1<V_1$, then the free product of $V_0$ and $V_1$ amalgamating the subgroups $U_i$ according to the isomorphism $\phi$ is, by definition, the quotient of the free product $V_0\ast V_1$ by the normal subgroup generated by $\{u^{-1}\phi(u) \mid u\in U_0\}$. The most important property of this construction is that the natural map from $V_i$ to the amalgamated free product is injective for $i=0,1$.
 The existence of a \textit{normal form} for elements of amalgamated free products is also crucial to many applications; 
  see \cite[p. 181]{Lyndon2001} for details.

From a modern perspective it is natural to define the amalgamated free product as the pushout in the category of groups of the diagram
\begin{equation}\label{amalgam}
V_0 \overset{i_0}\leftarrow U \overset{i_1}\rightarrow V_1
\end{equation} 
where $i_0$ (resp. ~$i_1$) is an injective homomorphism with image $U_0$ (resp. ~$U_1$) and $i_1\circ i_0^{-1} = \phi$. The somewhat sloppy notation $V_0\ast_U V_1$ is commonly used for this pushout,  despite the fact that it depends heavily on the map $\phi$; in contrast,  in \cite{HNN1949}
amalgamated free products are described in prose, without a dedicated notation, and the language is precise.  

\begin{remark}
Modern treatments of HNN extensions usually set them alongside   free
products with amalgamation, presenting these two constructions on a more or less equal
footing.  It is therefore striking, from
a historical perspective, to see how explicitly the proofs in the original HNN paper  build on
the theory of amalgamated free products.  
The final sentence of the introduction  is telling: {\em ``The principal tool throughout is the free product of two groups with amalgamated subgroups". \cite[p. 248]{HNN1949}}.  Such products were so firmly understood
in 1949 that their basic properties were quoted without proof, and Schreier's paper was
not referenced at all.
\end{remark}

\subsection{HNN extensions and their notation}\label{Subsec:HNN-background}
An  {\em HNN extension} is defined using two pieces of data:  a group $G$ and an isomorphism $\phi \colon A_0 \to A_1$ between two  subgroups $A_0, A_1$ of $G$. The \textit{HNN extension of $G$ associated to the isomorphism $\phi$,
with stable letter $t$} is the quotient of the free product $G\ast \<t\>$ by the normal subgroup generated by
$\{ t^{-1}a^{-1}t\phi(a) \mid a\in A\}$,  it is denoted by 
\[
G\ast_\phi.
\]
The two subgroups $A_0$ and $A_1$ are called the \textit{associated subgroups}. 

The notational conventions for HNN extensions are, in general, somewhat unsatisfactory. The notation that we use throughout this article, $G\ast_\phi$, is both standard and unambiguous.  However, it is common for authors to use the more informal notation $G\ast_{A_0}$, suppressing mention of the  isomorphism $\phi : A_0\to A_1$, either because it is not important to the task at hand,  or else because the author wants to focus attention on the subgroup. The more lax notation $G\ast_A$ is also common, where $A$ is a representative of the  isomorphism type of $A_0\cong A_1$; for example one might write $G\ast_\Z$ if $A_0$ and $A_1$ are infinite cyclic subgroups.   This last notation is convenient in contexts where one wants to make statements about all HNN extensions ``of the form'' $G\ast_A$,  for a fixed $G$ or for $G$ drawn from a class of groups, and for all $A<G$ of a given isomorphism type.  For example, one might want to say: {\em if $G$ is finitely presented and $A$ is finitely generated, then [any HNN extension of the form] $G\ast_A$ is finitely presented.} 

The notation for the multiple HNN extension associated to a family of isomorphisms $\phi_\sigma:A_\sigma\to B_\sigma \ (\sigma\in\Sigma)$ between subgroups of a fixed group $G$ (which we will encounter in Theorem II)  is less standardized; we shall use the notation $G\ast_{\{\phi_\sigma\}}$ for this.

\section{The 1949 HNN article}\label{Sec:HNN-article}

\noindent In this section, we will describe the contents of  the HNN article \cite{HNN1949} in more detail.  We have
retained the original phrasing of the four main theorems  as well as the original 
section headings, in order to faithfully reflect the authors' emphases and priorities.
%(The numbering of sections is changed from 2,3 etc. to  2.1, 2.2 etc. ).  

\subsection{Isomorphism between subgroups as conjugacy in a larger group}\label{Subsec:Thm1}

\begin{thmI}\label{ThmI}
Let $\mu$ be an isomorphism of a subgroup $A$ of a group $G$ onto a second subgroup $B$ of $G$. Then there exists a group $H$ containing $G$, and an element $t$ of $H$, such that the transform by $t$ of any element of $A$ is its image under $\mu$:
\begin{equation}\label{2.1}
t^{-1}at = \mu(a)  \text{   for all  } a\in A.
\end{equation}
\end{thmI}
\begin{proof}[Proof sketch.]
The authors of \cite{HNN1949} first take two copies of the free product 
$G\ast\Z$,
\[
K := G\ast \<u\>, \ \ \  L := G\ast\<v\>.
\]
Next, they argue that the subgroup $U<K$ generated by $G$ and $u^{-1}Au$ is the free product of $G$ and $u^{-1}Au$, and likewise the subgroup $V<L$ generated by $G$ and $vBv^{-1}$ is the free product of $G$ and $vBv^{-1}$.  It follows that there is an isomorphism which ``maps $G$ {\em qua} subgroup of $U$ identically onto $G$ {\em qua} subgroup of $V$" and maps  $u^{-1}au$ to $v\mu(a) v^{-1}$ for each $a\in A$.  The group $H$ described in the statement of the theorem is the  free product of $K$ and $L$ with the subgroups $U$ and $V$ amalgamated via this isomorphism,  with $t:=uv\in H$.
\end{proof}

Following this proof, the authors  note in a corollary that if $G$ is torsion-free\footnote{The quaint terminology of the time for groups with no non-trivial elements of finite order was ``locally infinite''.} then so is $H$. The concise proof of this fact illustrates how firmly basic facts about amalgamated free products were established at this time: {\em For a free product of two locally infinite groups, with or without amalgamated subgroups, is locally finite.} 

They then immediately
shift attention from the group $H$ in the proof of Theorem I to the subgroup $H_1$ generated by $G$ and $t$,
which is a more canonical choice for the group whose existence is asserted in the theorem. 
It is in this sentence that the HNN extension is born:  $H_1$ is what we now call the  HNN extension of $G$ associated to 
the isomorphism $\phi:A\to B$. 
The description of $H_1$ in the following sentence is relied upon  heavily  in later sections:   ``The defining relations of $H_1$ are those of $G$ and relations (2.1)" (their (2.1) being our (\ref{2.1})).  No notation is dedicated to this group in \cite{HNN1949}, but subsequently the notation  $G\ast_\phi$ became fairly standard, as discussed briefly in \S\ref{Subsec:HNN-background}.

To the modern eye, it is obvious and noteworthy that $G\hookrightarrow H_1$ is {\em universal} in the sense that  if
 $i:G\hookrightarrow \G$ is an embedding of groups and there exists $\tau\in\G$  with $\tau^{-1}i(a)\tau = i\mu(a)$ for all $a\in A$,
 then $i$  factors through $G\hookrightarrow  H_1 = G\ast_\mu$.  This fact is not remarked upon explicitly in \cite{HNN1949}, but there is no doubt that the authors were aware of it, indeed this issue is explored in some detail in B.H.~Neumann's paper \cite{Neumann1943b} and the authors of \cite{HNN1949} translate Theorem I into the language of that paper: this translation is in Theorem ${\rm{I}}^\prime$,  which is convenient for the discussion of multiple HNN extensions in \S3.

In Theorem I, if $G$ belongs to a class of groups that is closed under the taking of subgroups, then one might ask if the group $H$ can also be taken to lie in this class.  There is a substantial literature addressing this 
question.  In the class of finite groups,  the embedding always exists:
if  $G$ is finite and  $a, b \in G$ have the same order, then one can embed $G$ in a {\em finite} group $H$ 
so that $a$ and $b$ are conjugate in $H$.  This is  an easy consequence of the fact that HNN extensions of finite groups are residually finite   \cite{Baumslag1962}.

\subsection{Infinite groups with only two classes of conjugate elements}\label{Subsec:TheoremsII-III}

Following Theorem I,  the authors of \cite{HNN1949} quickly generalize their main construction to cover families of isomorphic subgroups (Theorem II),  a generalization that is needed for the two applications of historic importance that they give. Continuing their discussion in the language of adjunctions, the authors prove that any family of HNN extensions can be carried out simultaneously, and that the stable letters used to form these extensions generate a free subgroup of the enlargement of $G$ that is obtained. Specifically, they prove the following:
 
 \begin{thmII}
 Let $\mu_\sigma$ (where $\sigma$ ranges over an index set $\Sigma$) be an isomorphism of a subgroup
 $A_\sigma$ of a group $G$ onto a second subgroup $B_\sigma$. Then there exists a group $H$ containing $G$, and containing also a group $T$ freely generated by a set of elements $t_\sigma\ (\sigma\in\Sigma)$, such that for any $\sigma$ in $\Sigma$ the transform by $t_\sigma$ of an element in $A_\sigma$ is its image under $\mu_\sigma$ :
 \begin{equation}\label{Eq3.1}
 t_\sigma^{-1} a_\sigma t_\sigma = \mu_\sigma(a_\sigma)  \text{   for all   } \sigma\in\Sigma, \ \ \ a_\sigma\in A_\sigma.
 \end{equation}
 \end{thmII}
\begin{proof}[Proof sketch.]
 The proof of the authors begins as follows: ``We define $H$ to be the group generated by $G$ and elements $t_\sigma, \ \sigma\in\Sigma$, subject only to the relations defining $G$ and to \eqref{Eq3.1}."  After a brief argument to reduce to the case where $\Sigma$ is finite, they invoke ``a finite number of applications" of Theorem ${\rm{I}}^\prime$, and arrive at the statement of the theorem.
\end{proof}
 
We emphasize that the focus in the proof is again on the universal choice for $H$, which we today would call the (multiple) HNN extension associated to the given data. As a first application of Theorem II, the authors prove that \textit{every torsion-free group $G$ can be embedded in a group $G^*$ in which each pair of non-trivial elements of $G$ are conjugate} (Corollary~1).  
This is proved by noting that by taking $\Sigma$ to be the set of all pairs of elements  $(a, b)$ such that $a$ and $b$ have the same order in $G$, one obtains an embedding $G\hookrightarrow G'$ such that every pair of elements of $G$ that have the same order become conjugate in $G'$ (which is the minimal choice of $H$, i.e.\ the multiple HNN extension). By stating Corollary 1 and introducing the temporary notation $G'$ for the multiple HNN extension, the authors gave themselves the vocabulary to present a concise and transparent proof of the following theorem;  this thoughtful and concise style of presentation runs throughout the article.  Corollaries 2 and 3 are the observations that if $G$ is torsion-free or countable, then $G'$ will have the same property.

Section 2 of \cite{HNN1949} ends with the first of the two major applications of Theorem I. 

 \begin{thmIII} Any locally infinite group $G$ can be embedded in a group $G^*$ in which all elements except the unit element are conjugate to each other.
 \end{thmIII}
 \begin{proof}[Proof sketch.] 
 To prove this,  the authors define a sequence of groups 
 \[
 G=G_0\subset G_1\subset G_2\subset\dots
 \]
 by setting $G_{n+1}=G_n'$, ``in the notation of Corollary 1", and then form the union 
 \[ 
 G^* = \bigcup_n G_n
 \]
 `{\em in the usual ``group tower" sense}', i.e.\ they take a direct limit. This group is torsion-free by a previous corollary, and since  the non-trivial elements of $G_n$ are all conjugate in $G_{n+1}$,  the non-trivial elements of $G^*$ form a single conjugacy class.
 \end{proof}

In a corollary to Theorem III, the authors note that if $G$ is countable then $G^*$ as constructed in the proof is also countable.  
%They also note that the only finite group with exactly two conjugacy classes of  elements is the cyclic group of order $2$.

\begin{corollary}\label{Cor:TF-infinite-simple-exist}
There exist torsion-free,  infinite,  countable groups that are simple.
\end{corollary} 
 
The authors do not repeat the observation from their introduction that their construction gives the first examples of groups with the above property, but they do note that by varying the initial group $G$ one obtains continuously many pairwise non-isomorphic groups $G^*$. It seems likely that the authors discussed whether their proof of Theorems II and III could be adapted to construct a \textit{finitely generated} infinite group in which all non-trivial elements are conjugate. Such groups do exist, but the technology required to prove this was not available in 1949: the first  examples were constructed by D.\ Osin in 2010 \cite{Osin2010} using small cancellation techniques and the theory of relatively hyperbolic groups. 

The first finitely generated infinite simple groups were 
  constructed by Graham Higman \cite{Higman1951} in 1951. Using HNN extensions and
  amalgamated free products,  he exhibited a group with four generators and four defining relations that is infinite
  but has no non-trivial finite quotients; the quotient of such a  group by a maximal proper normal 
  subgroup is finitely generated,  infinite, and simple.  
In 1953, Higman's doctoral student R.\ Camm \cite{Camm1953} exhibited  uncountably many finitely generated infinite simple groups that decompose as amalgamated free products of free groups $F_2 \ast_{F_\infty} F_2$. 
R.\ Thompson (in unpublished notes) found the first examples of finitely presented infinite simple groups in 1965 (cf.\ also Higman \cite{Higman1974}).  Using geometric methods (groups acting freely on products of trees),  in 2000 Burger--Mozes \cite{Burger2000} constructed infinite families
  of finitely presented infinite simple groups that decompose as amalgamated free products 
  $F_n \ast_{F_m} F_k$.

\subsection{Embedding countable groups in two-generator groups}\label{Subsec:TheoremIV}

The next result in \cite{HNN1949}
 is the famous theorem that every countable group can be embedded in a group that requires only two generators, and furthermore that this enlargement of the group can be achieved without increasing the number of defining relations. 
 The analogous result for
 semigroups is considerably easier, and  had  been obtained earlier by Marshall Hall Jr. \cite{Hall1949}. Given a semigroup $S$ generated by   $g_1, g_2, \dots$, the map 
\begin{equation}\label{Eq:semigroup-embedding}
g_i \mapsto baba^{i+2}baab
\end{equation}
embeds $S$ in the semigroup on two generators $a, b$ with the same defining relations as $S$, rewritten as words in $\{ a, b\}$
using the given substitution (cf.\ \eqref{g_i} below).   

\begin{thmIV}
Any countable group $G$ can be embedded in a group $H$ generated by only two elements. If the number of defining relations for $G$ is $n$, the number of defining relations for $H$ can be taken to be $n$.
 \end{thmIV}
 \begin{proof}[Proof sketch.] 
 
 The proof begins with a generating set $g_1,g_2,g_3,\dots$ for $G$.  One reduces to the case where
the generators all have infinite order by replacing $G$ with $K=G\ast\Z = G \ast \<u\>$; let $u_i=g_iu$. A multiple HNN extension of $K$,  denoted by $P$, is considered; this has stable letters $t_1, t_2,\dots$ with $t_i^{-1}ut_i:=u_i$ for all $i$. Theorem II assures us that the natural map from $K$ to $P$ is injective and that the elements $t_i$ freely generate a free subgroup $T<P$.  Note that $P$ is generated by $\{u, t_1, t_2,\dots\}$.  A key point to observe is that by killing the $t_i$ we obtain a retraction $P\to \<u\>$.

The next step in the construction achieves the aim of embedding $G$ in a finitely generated group: the authors take a free group $F_1$ freely generated by two new symbols $b,v$,  fix a basis for a subgroup of infinite rank -- $\{s_i = b^{-i}vb^i \mid i=1,2,\dots\}$ is a convenient choice -- and then form the free product with amalgamation that identifies $s_i$ with $t_i$ for $i=1,2,\dots$; let $Q$ denote the resulting group. A key point to observe is that $Q$ is generated by $\{u,b,v\}$.  Another key point is that $F_2=\<b,u\><Q$ is free of rank $2$, as one sees by extending the retraction $P\to \<u\>$ to an epimorphism $Q\to\<u\>\ast\<b\>$ that maps  $v$ trivially.  To complete the proof,  $H$ is defined to be the HNN extension with stable letter $a$ corresponding
to the isomorphism $\mu: F_1\to F_2$ defined by $b\mapsto u,\  v\mapsto b$.  As
$a^{-1}ba=u$ and $a^{-1}va=b$,  we see that $H$ is generated by $\{a,b\}$. This proves the first part of the theorem.

For the second part, we note that the groups constructed at the different stages of this proof are 
\begin{equation}\label{buildH}
K= G\ast \Z,\    \  \  P=  K\ast_{\{\phi_i\}},  \  Q= P\ast_L F,\  \  \  H= Q\ast_{\mu},
\end{equation}
where the  family of isomorphisms $\phi_i:\<u_i\>\to \<u\>$ and the rank of the free group $L$ are finite if and only if the  original generating set $g_1,g_2,\dots$ for $G$ is finite. The authors are careful to keep track of the defining relations of these intermediate groups, and  this careful bookkeeping enables them to prove the assertion about the number of relations needed to present $H$.  Specifically,  their argument begins with a choice of generating set $g_1, g_2, g_3, \dots$ for $G$, and it produces a two-element generating set $\{a,b\}$ for $H$ so that, for all $i$,
\begin{equation}\label{g_i}
g_i = a^{-1}b^{-1}ab^{-i}ab^{-1}a^{-1}b^ia^{-1}bab^{-i}aba^{-1}b^i.
\end{equation}
(It is worth comparing \eqref{g_i} with  \eqref{Eq:semigroup-embedding}). It follows that from each defining relation in $G$, say $R_\lambda(g_1,g_2,\dots)=1$, we obtain a valid relation $R_\lambda(a,b)=1$ in $H$ by formally substituting  the $\{a,b\}$-words of (\ref{g_i}) in place of $g_1,g_2,\dots$.  From this, one concludes that if $\< g_1,g_2,\dots \mid R_1,\dots, R_n\>$ is a presentation of $G$, then $\< a, b \mid R_1',\dots, R_n'\>$ is a presentation of $H$.
 \end{proof}

The authors end \cite{HNN1949} with an important remark: there cannot exist a single two-generator group that contains every countable group because a finitely generated group has only countably many finitely generated subgroups, whereas there are uncountably many pairwise non-isomorphic two-generator groups \cite{Neumann1937}. Interestingly, they use the word ``universal" (their quotation marks) to describe the putative containing group that they have observed does not exist. They do not comment on the fact that this argument with cardinalities does not extend to finitely presented groups, since there are only countably many such groups up to isomorphism, but one imagines that they discussed it.  The existence of a ``universal" finitely presented group  i.e.~a finitely presented group that contains an isomorphic copy of every finitely presented group, was established by Higman twelve years later \cite{Higman1961}. We will discuss this in greater detail in \S\ref{Subsec:Higman-embedding}.

Finally, in a note added in proof, the authors recount that they learned from Hirsch that the basic problems solved in \cite{HNN1949} were raised in the 1944 Russian edition of Kurosh's textbook \cite[pp.~356-357]{Kurosh1944ed1}. This discussion was removed in subsequent editions of Kurosh's textbook, including the influential English edition, and was replaced by a brief discussion of HNN extensions in the 1956 edition \cite[Chapter IX--X]{Kurosh1956}. Intriguingly, the authors report that Freudenthal had obtained Theorems IV and V independently; it seems that that work was never published.  

We conclude our discussion of \cite{HNN1949} with three brief remarks.

\begin{remark}[Finiteness Properties]
The repeated emphasis  of the authors on the economical nature of their constructions -- i.e. the smallness of the extensions formed -- is noteworthy and it has remained an important feature of embedding theorems up to the present day. It is clear that for Higman, Neumann and Neumann, being countable was an important measure of smallness,  as were finite generation and finite presentation. In Theorem IV, if one begins with a finitely presented group $G$,  
then the finite presentability of $H$ can be gleaned immediately from (\ref{buildH}). The way in which one deduces finite presentation from (\ref{buildH}) extends readily to higher finiteness properties, which did not become 
a focus of attention until group theory became entwined with algebraic topology in the 1960s.

For each integer $d\ge 0$ one defines a group $G$ to be {\em of type ${\rm{F}}_d$} 
if there is a connected CW-complex $X$ with  fundamental group $G$ such that the universal cover of $X$ is contractible, and $X$ has only finitely many $k$-cells in dimensions $k\le d$.  A group is finitely generated if and only if it is of type ${\rm{F}}_1$, and it is finitely presented if and 
only if it is of type ${\rm{F}}_2$.  A homological variant of this definition defines the weaker conditions ${\rm{FP}}_d$. A standard, straightforward argument shows that all of these conditions are preserved by amalgamated free products and HNN extensions in which the amalgamated (or associated) subgroups are finitely generated and free. Thus it follows from  (\ref{buildH}) that the embedding $G\hookrightarrow H$ in Theorem IV preserves all of these properties. Likewise, any other property that is preserved under such combinations over free groups will be transmitted from $G$ to $H$, for example the property of being torsion-free,  or of having only finitely many conjugacy classes of elements of finite order.
\end{remark} 

\begin{remark}[The computable nature of the construction]
It is worth reflecting on the explicit, algorithmic nature of the construction in Theorem IV:
given a presentation of $G$,  from (\ref{g_i}) one immediately obtains  an explicit presentation of $H$ and an explicit  embedding $G\hookrightarrow H$. The authors of \cite{HNN1949} do not comment directly on the explicitness of the construction, but they do shine a 
light on it by reinterpreting Theorem IV as a striking statement (Theorem V)
about the normal subgroup of the free group $F(a,b)$
generated by the words  appearing on the right-hand side of (\ref{g_i}). Later, in 1961, Higman would use the computability of the embedding $G \hookrightarrow H$ in the proof of one of the corollaries to his  \textit{embedding theorem}, which we discuss in \S\ref{Subsec:Higman-embedding}.
\end{remark}

\begin{remark}
It is noteworthy that  at various points in \cite{HNN1949} it is emphasized that the basic operations of HNN extension and amalgamated free product can be performed multiple times, simultaneously or iteratively; indeed this is a crucial feature of the applications that are given. This multiplicity foreshadows Serre's notion of a {\em graph of groups}, which we will discuss in detail in \S\ref{Subsec:Bass-Serre-theory}.
\end{remark}

\section{Constructing new groups with HNN extensions}\label{Sec:Constructing-new-groups}

\noindent We turn now to  the impact of HNN extensions in the subsequent literature. In this section, we will focus on how HNN extensions were used in combinatorial group theory to \textit{construct} finitely presented groups with various rather wild properties, including groups with undecidable word problem and finitely presented groups that are universal in the sense that
they contain an isomorphic copy of every finitely presented group.

\subsection{The word problem}\label{Subsec:word-problem}

An early and important  application of HNN extensions was in providing simplifications and conceptual explanations for an array of disparate constructions related to the \textit{word problem} for finitely presented groups. We recall that this problem, whose fundamental importance was emphasized by Dehn \cite{Dehn1912},  asks for an algorithm that, given two words $u, v$ in the generators of a group,  decides whether or not $u$ and $v$ represent the same element. One of the key driving forces behind the application of HNN extensions in this setting was John L.\ Britton. Today,  what is regarded \textit{the} crucial lemma concerning HNN extensions bears his name, and we shall briefly sketch how this came to be. Britton had been heavily influenced by the Neumanns from an early stage in his career, having completed a master's degree under H.\ Neumann at the University College of Hull in 1951, and a PhD under B.\ H.\ Neumann at the (Victoria) University of Manchester in 1953. In his thesis, the word problem plays a central role, but there is no mention of HNN extensions; instead, Britton extends much of the work of Tartakovskii \cite{Tartakovskii1949a, Tartakovskii1949b} on early small cancellation theory, via a series of technical lemmas. The key focus here lies in \textit{solving} the word problem in particular classes of groups, and the general undecidability of the word problem, announced in 1952 by P.\ S.\ Novikov \cite{Novikov1952}\footnote{It bears mentioning that the group specified in Novikov's 1952 announcement paper \cite{Novikov1952} may actually have had \textit{decidable} word problem, much to the dismay of Boone, see \cite[p. 169]{Collins1985},  but a different group appeared in the full 1955 article \cite{Novikov1955}.}, was still some way from becoming a generally accepted truth.\footnote{Indeed, Britton \cite[p. 9]{Britton1953} mentions that ``[Novikov's] proof does not yet seem to have found general acceptance''. This skeptical tone may have originated with B.\ H.\ Neumann, who in the early 1950s had claimed a proof of the general \textit{decidability} of the word problem for groups, at the same time as A.\ Turing claimed a proof of the \textit{undecidability} of the same problem \cite[p. 169]{Collins1985}. Both later retracted their claims.}

A few years later, Britton would return to the word problem in two articles. At the same time, P.\ S.\ Novikov \cite{Novikov1955} and W.\ Boone \cite{Boone1957} had, in 1955 and 1957, respectively, published their independent proofs of the general undecidability of the word problem for groups.  Britton had also arrived at this conclusion in 1957, independently of Boone's work, but drawing inspiration from Novikov's article (which he read in the original Russian, cf.\ \cite[Footnote 3]{Boone1959}). One of the drawbacks of the arguments of  Boone and Novikov was that they both failed to display a \textit{conceptual} reason for the undecidability of the word problem. Indeed, reading the articles today, one is quickly overwhelmed by the heavy symbolic notation designed for the specific theorem in mind; one must take a plunge into the argument at the start of the article, and endure to the very end before gaining any satisfaction. This shortcoming was noted by R.\ Fox (in 1955) and by Higman (1957), who had already pointed out to Boone,  in personal communication,   that several of his arguments could be simplified by the use of HNN extensions (see \cite[\S37]{Boone1959}). This first step towards a more conceptual proof forms part of what Britton did in his first article on the undecidability of the word problem. A key part of his argument was the following, which essentially\footnote{The principal lemma, as stated in \cite{Britton1958}, contains the additional but ultimately superfluous assumption that the associated subgroups intersect trivially. Britton's first complete statement of Britton's Lemma, as given in this article, is from his second article \cite[Lemma~4]{Britton1963}, see below.} forms the second half of what he calls the ``Principal Lemma'' \cite[Lemma~4(2)]{Britton1958}, 
 a result which today bears his name. 

\begin{blemma}\label{BLemma}
Let $G^\ast$ be an HNN extension of a group $G = \pres{}{A}{R}$ with stable letter $t$ and associated subgroups $H_1, H_2$, with isomorphism $\varphi \colon H_1 \to H_2$. If $w$ is a freely reduced word in $A^{\pm 1}$ and $t^{\pm 1}$ such that $w = 1$ in $G^\ast$, then either:
\begin{enumerate}[label=(\roman*)]
\item $w$ contains no occurrence of $t$ or $t^{-1}$, and $w = 1$ in $G$; or
\item $w$ contains a subword of the form $t^{-1} u_1 t$, where $u_1 \in H_1$; or
\item $w$ contains a subword of the form $t u_2 t^{-1}$, where $u_2 \in H_2$.
\end{enumerate}
\end{blemma}
\begin{proof}[Proof sketch.]  
Britton deals first with the case $H_1 = H_2$,
where he observes that $G^\ast$ is  the free product of $G$ and $H_1 \rtimes_\phi \Z$,  
amalgamated over the common subgroup $H_1$,  and therefore
 the normal form theorem for amalgamated free products can be applied.
He reduces the general case to this special case by manipulating the double HNN extension of $G$ in which there are two stable letters,  each conjugating $H_1$ to $H_2$ via the isomorphism $\phi$.
\end{proof} 

Britton made good use of this lemma to simplify arguments concerning equalities and non-equalities of certain words in the group constructed by Boone and Novikov. However,  the starting point of the refined argument, and many of its technical steps, remained fundamentally non-algebraic,  concentrating instead on the undecidability of a certain problem concerning Turing machines, similar to the approach by Turing to prove undecidability of the word problem in finitely presented cancellative semigroups \cite{Turing1950, Boone1958}.  Britton's proof
also suffered from a heavy technical step involving the \textit{decidability} of the word problem in certain groups,  falling back on his earlier work \cite{Britton1953,Britton1957a, Britton1957b} on small cancellation theory to do so. 

In 1963, Britton remedied both of these flaws, and in doing so finally produced what may be called the first approachable proof of the undecidability of the word problem for groups. His starting point is algebraic, beginning with the existence of a finitely presented semigroup $S$ with undecidable word problem, as shown by Markov \cite{Markov1947} and Post \cite{Post1947}. This semigroup $S$ is then encoded into a group, which is used, by a sequence of HNN extensions, to construct a finitely presented group $G$ such that the word problem for $G$ is equivalent to that of $S$, yielding undecidability. Boone's 1957 argument had the same starting point, but the passage through HNN extensions and Britton's Lemma transforms this argument into a fully group-theoretic one.  Britton also makes a key, and seemingly original, distinction between \textit{combinatorial} arguments versus \textit{group-theoretic} arguments \cite[p. ~16]{Britton1963}; in modern language, the distinction can be summarized as being arguments about a presentation of a group versus arguments about the group itself. A curious fact is that Britton does not cite the HNN-article \cite{HNN1949} but instead refers to the book of Kurosh \cite[Chapter IX--X]{Kurosh1956} on group theory for the HNN theorem as well as the general theory of amalgamated free products.

The undecidability of many other algorithmic problems in group theory also make heavy use of HNN extensions, and Britton's Lemma, to simplify the heavy combinatorial arguments. We mention only one other instance, to highlight the dramatic simplifications that can take place. The \textit{Adian--Rabin theorem} states, roughly speaking, that for many ``reasonable'' group properties $\mathcal{P}$ (including being trivial, being infinite, being simple, being abelian, ...), there is no algorithm which takes as input a finite group presentation, and decides whether or not the group defined by that presentation has property $\mathcal{P}$ or not. This implies, among other things, that the isomorphism problem for finitely presented groups is undecidable. Both Adian and Rabin had as a starting point the existence of a finitely presented group with undecidable word problem. Starting from this, the proof by Rabin \cite{Rabin1958} covers around 10--15 pages, leaving some details aside, and leveraging amalgamated free products; the proof by Adian \cite{Adian1955, Adian1957} (cf.\ translations and commentary in \cite{NybergBrodda2022}) is essentially  self-contained, and covers almost 60 pages of rather brutal combinatorial details. By contrast,  by using HNN extensions  and Britton's Lemma, one can give  a conceptual proof of the Adian-Rabin theorem,   easily understandable to a modern student,  can be given in less than two pages, as in Baumslag's textbook \cite[pp.\ 112--113]{Baumslag1993} or Miller's survey \cite[pp. 13--16]{Miller1992}.

\subsection{The embedding theorem of Higman}\label{Subsec:Higman-embedding}

In  section \S\ref{Subsec:word-problem}, we discussed how HNN extensions were used to simplify the proof of a deep theorem
and make it more conceptual.  A similar scenario played out in the case of the \textit{Higman Embedding Theorem}.
This theorem, proved in 1961, 
 was a crowning achievement of combinatorial group theory  in the twentieth century.
In this section, we will give a brief sketch of the history of this theorem,  up to the appearance of the standard modern proof,
which is due to Aanderaa \cite{Aanderaa1973}. 

A finitely generated group  $G$ is said to be \textit{recursively presented} if it admits a presentation $\pres{}{A}{R}$ in which $A$ is finite and $R$ is a recursively enumerable subset of the free group $F(A)$ on $A$.   
Thus any finitely presented group is recursively presented. 
Moreover, any finitely generated subgroup of a finitely presented group is recursively presented: if $H<G$ is generated by a finite subset $A$,  then $G$ has a finite presentation of the form 
$G=\<A\cup B \mid S\>$,  and $H=\<A\mid R\>$ where the list of words  $R\subset F(A)$ is obtained
by enumerating all finite products in $F(A\cup B)$ of conjugates  of elements from $S$,  
freely reducing each product and adding it to $R$ if the free reduction contains only letters from $A$.
Higman's remarkable theorem tells us that all finitely generated,  recursively presented groups arise in this way.

%\marginpar{this is Higman's wording}
\begin{HEtheorem}[{\cite[Theorem~1, p.\ 456]{Higman1961}}] 
A finitely generated group  can be embedded in a  finitely presented group if and only if it  is recursively presented.
\end{HEtheorem}
 
Higman's proof of the Embedding Theorem  is long and technical,  so we will only give a brief overview.  The first 
part of the proof is devoid of group theory: it consists of a characterization of the (recursively enumerable) set $E$ of functions $f \colon \Z \to \Z$ with finite support,  building it out of smaller sets by a simple set of operations, in the spirit of Kleene. Next, the notion of ``benign'' subgroup is introduced:  a subgroup $H$ of a finitely generated group $G$  is \textit{benign} if the HNN extension $G\ast_H$ corresponding to the identity map ${\rm{id}}_H$  can be embedded in a finitely presented group. This condition is shown to be equivalent, via a series of arguments involving amalgamated free products and Theorem I
of \cite{HNN1949} to the existence of a finitely presented group $K$ that contains $G$ and has a finitely generated subgroup $L$ 
with $G \cap L = H$.   Higman  links these preparatory constructions by exhibiting 
 a correspondence between the elements of the set
$E$ and the elements of a free group $F$, and between subsets $X \subset E$ and certain subgroups $A_X \leq F$; the
correspondence is such that  $X$ is recursively enumerable if and only if $A_X$ is benign.  In the final section of the paper,
he uses this correspondence to prove the Embedding Theorem, prefacing the remaining lemmas with a sigh of relief: ``{\em We have now done all the hard work, and it is relatively easy to complete the
proof."} 

From the standpoint of the current article, it is worth emphasizing that Higman's proof makes liberal and crucial 
use of HNN extensions, but in a rather more convoluted way than one might expect.
Moreover,  a significant portion of his paper is devoted to purely recursive function-theoretic arguments.
In 1968, Valiev \cite{Valiev1968} simplified Higman's proof by starting with a characterization of recursively enumerable sets
that had arisen from work on Hilbert's 10th Problem.; this removed the need for Higman's analysis of the set $E$.  In 1970 Matiyasevich settled Hilbert's 10th problem by showing
that all recursive subsets of $\Z^n$ are Diophantine \cite{Matiyasevich1970},  and if one takes this as a starting point,  one can follow
Valiev's argument to simplify his proof further.   This is the approach taken in Section IV.7 of 
Lyndon and Schupp's book \cite{Lyndon2001}. Finally, in 1973, S.\ Aanderaa gave a purely group-theoretic proof of Higman's Embedding Theorem, which begins with a simple encoding of an arbitrary Turing machine into a (semi)group presentation, and proceeds to build the Higman embedding through a sequence of HNN extensions and amalgamated free products. This is the proof that is most often given and adapted in modern treatments of
the Higman Embedding Theorem and refinements of it.

In the introduction to his paper, immediately after stating the Embedding Theorem,  
Higman generalizes the non-trivial direction of the theorem to cover recursively presented groups that
are not finitely generated, and this generalization is needed for the  applications that we want to discuss now.
  By definition, a recursive presentation of a countable group $G$ consists of a free group $F$ of
finite or countable rank,  with an enumeration of a basis $a_1,a_1,\dots$ and an epimorphism $F\twoheadrightarrow G$
whose kernel is the normal closure of a recursively enumerable subset of $F$; if such a presentation exists, then $G$ is said to be recursively presented.  Theorem II in the HNN paper \cite{HNN1949} provides
an embedding of $G$ into a 2-generator group $G^*$,  and Higman notes that 
if $G$ is recursively presented,   the embedding is effective and  $G^*$ is recursively presented, so he can apply his 
theorem to $G^*$.

\begin{corollary}\label{c:not-fg}
Any recursively presented group can be effectively embedded in a finitely presented group.
\end{corollary} 

With this corollary in hand, it is easy to deduce the  existence of a finitely presented group with undecidable word problem. 
The proof of this result via Higman's Embedding Theorem is rightly regarded as the epitome of conceptual proofs of the existence of such groups; the fact that the deduction
 is so easy demonstrates the remarkable power of Higman's theorem.  The salient point here is  that
recursively presented groups with pathological properties are much easier to construct than finitely presented ones; groups
with unsolvable word problem are a case in point.  
 Higman \cite[p. 456]{Higman1961} notes that recursively presented groups with unsolvable word problem are already in the
 literature (e.g.~\cite{Britton1958}), but gives his own construction, as follows. Let $f \colon \mathbb{N} \to \mathbb{N}$ be any recursive function whose range is not a recursive set, examples of which are easy to construct \cite[Theorem~4.6]{Davis1958}, and let $G$ be generated by $a, b, c, d$  subject to the defining relations
\[
b^{-f(n)} a b^{f(n)} = d^{-f(n)} c d^{f(n)} \quad (n = 0, 1, 2, \dots).
\]
Then $G$ is recursively presented and is an amalgamated free product of the form $F_2\ast_{F_\infty}F_2$.
The normal form theorem for amalgamated free products tells us that
\[
b^{-m} a b^m = d^{-m} c d^m
\]
if and only if $m$ lies in the range of $f$, and hence the word problem for $G$ cannot be solved (because membership of the range of $f$ cannot be decided). Thus $G$ is a recursively presented group with undecidable word problem. And then, by the embedding theorem, we can now embed $G$ into a finitely presented group $G'$.  
The second important application of the Embedding Theorem that we want to 
highlight concerns the existence of universal finitely presented groups.

\subsection{Universal groups and other corollaries}\label{Subsec:universal-groups}
It is easy to enumerate the set of all finite presentations with generators drawn from a fixed alphabet $x_1, x_2, \dots$, and by from such an enumeration one can derive a recursive presentation of the free
product  $\Conv_{\substack{i \geq 1}} G_i$ of the groups $G_1, G_2, \dots$ given by these finite presentations (where we pay no attention to which of the $G_i$ are isomorphic).
 As before, by the HNN Theorem we can embed this countable group into a $2$-generator group in an effective way, which can in turn be embedded into a finitely presented group $G$ by the Higman Embedding Theorem. But then $G_i \leq G$ for every $i \geq 1$. In particular, we have just shown that \textit{there exists a finitely presented group containing an isomorphic copy of every finitely presented group}. Such groups are often called \textit{universal} (finitely presented) groups. The almost magical ease of this construction further underscores the remarkable power of Higman's theorem.

Higman's Embedding Theorem remains one of the most powerful tools in group theory.
We mention just a few of the many subsequent developments that it engendered. First, given a recursively presented group $G$ with solvable word problem, it is natural to ask whether it can be embedded into a finitely presented group with decidable word problem; and analogously for other Turing degrees of the word problem. This turns out to be true, in the sense that Higman's Embedding Theorem can be shown to be made to preserve the Turing degree of the word problem \cite{Clapham1964, Valiev1975}.  More geometric measures of the
complexity of the word problem have also been intensively studied \cite{mrb2002},  foremost among them isoperimetric functions. The most definitive result in this direction is due to Birget, Ol'shanskii, Rips \& Sapir \cite{Birget2002}, who proved that the word problem of a finitely generated group $G$ is decidable in non-deterministic polynomial time if and only if the group can be embedded in a finitely presented group that has  a polynomial isoperimetric function. 

With the  existence of universal finitely presented groups in hand, many questions about the nature of such groups suggest themselves.  We know from Theorem II that 2-generator examples exist, but what is the minimal number of 
defining relations needed?  Magnus's solution to the word problem for 1-relator groups \cite{Magnus1932} implies that one needs at least two,  but no better lower bound is known. For upper bounds, Boone \& Collins \cite{Boone1974} constructed an example that needs 26 relations, and Valiev \cite{Valiev1977} constructed an example with 21 relations; no better upper bound is known. It is also fruitful to ask what additional properties one can require of a universal group.  
For example,  Baumslag, Dyer and Miller \cite{BDM}  proved that there exist finitely presented universal groups $U$ that are acyclic,  i.e.~$H_i(U,\Z)=0$ for all $i\ge 1$, and they used this to show that an arbitrary
sequence of countable abelian groups $(A_n)$,  with $A_1$ and $A_2$ finitely generated,  will arise as the homology sequence 
$H_n(G, \Z)$ of some finitely presented group $G$,  provided that the $A_n$ can be described in an untangled recursive manner.  In \cite{mrb:baumslag} it is shown that one can simultaneously require that $U$ has no non-trivial finite quotients.

Finally,  although Higman's proof of the Embedding Theorem is effective,  it leaves open the challenge of finding ``natural" embeddings of  particular recursively presented groups.
For example,  it is easy to give a recursive presentation of the additive group of the rationals $(\Q, +)$, so one might expect there to be an explicit natural embedding into a finitely presented group.  This particular 
challenge, which was a personal  favourite of Higman,  was solved in 2022 by Belk, Hyde and Matucci \cite{Belk2022},
but the analogous challenges for many other groups remain open.
%where $(\Q, +)$ is shown to embed into a two-generator, four-relation group, cf.\ also Mikaelian \cite{Mikaelian2023}.

\section{Decomposing groups: topology,  splittings,  and graphs-of-groups}\label{Sec:Bass-Serre}

\noindent In the previous section, we discussed how HNN extensions and amalgamated free products allow 
one to build interesting groups.  This perspective is in keeping with the emphasis of \cite{HNN1949} on embedding 
given groups into larger groups with desirable properties.  But there is a dual perspective that is equally important to emphasize,
namely the idea that when a group of interest arises, one wants to break it into elementary pieces, preferably in a canonical way.  
In this section, we will discuss how HNN extensions and amalgamated free products emerge as basic operations of decomposition
for groups that arise in topology. This will lead us to a discussion of {\em Bass-Serre theory}, graph-of-groups decompositions, and 
group actions on trees, which is the setting in which HNN extensions are studied in contemporary mathematics.

\subsection{van Kampen's Theorem}\label{Subsec:cutting-up-spaces}
The origins of combinatorial group theory are deeply entwined with the early history of the fundamental group in topology,
and the central role that HNN extensions play in the decomposition of groups emerges from  \textit{van Kampen's Theorem}, 
which we shall now recall.   Let $X$ be a path-connected topological space. If $X$ can be expressed as the union of two open,  path-connected subspaces $A$ and $B$ so that $C=A\cap B$ is path-connected,
%\footnote{To avoid restricting to the case where $C$ is connected, one has to talk about groupoids instead of groups} 
then we can choose a basepoint $x_0\in C$ and van Kampen's Theorem then states that the fundamental group $\pi_1(X,x_0)$ is the pushout of the diagram of groups
\begin{equation}\label{vK}
\pi_1(A,x_0) \leftarrow \pi_1(C,x_0)  \rightarrow \pi_1(B,x_0) 
\end{equation} 
where the homomorphisms are induced by the inclusions $A\hookleftarrow C \hookrightarrow B$.   
In general these homomorphisms will not be injective, but if they are then, comparing with \eqref{amalgam}, we deduce that $\pi_1X$ is an amalgamated free product of $\pi_1A$ and $\pi_1B$ with amalgamated subgroup $\pi_1C$. Colloquially, we can summarize this by saying that if cutting  along $C$ decomposes $X$ into two pieces, and $C\hookrightarrow X$ is {\em $\pi_1$-injective} -- i.e.\ the map $\pi_1C\to \pi_1X$ induced by inclusion is injective -- then $\pi_1X$ decomposes as an amalgamated free product of the fundamental groups of the two complementary pieces. 

This discussion raises the question of what happens when we cut a space along a $\pi_1$-injective subspace $C\hookrightarrow X$ that does not separate $X$. The answer, which follows from van Kampen's Theorem, is that provided $C$ sits nicely in $X$ (e.g.~if $C$ has a neighbourhood homeomorphic to $C\times (0,1)$),  then $\pi_1X$ decomposes an HNN extension of the form $G\ast_H$ where $H$ is the fundamental group of $C$ and $G$ is the fundamental group of $X\smallsetminus C$.  More generally, using the language of \textit{graphs of groups} discussed in the next section,  
if one cuts $X$ along a family of nicely embedded, disjoint, $\pi_1$-injective subspaces $C_\lambda\ (\lambda\in\Lambda)$, then $\pi_1X$ is isomorphic to the fundamental group of the graph of groups that encodes this decomposition: the unoriented edges of the graph are indexed by $\Lambda$, there is a vertex for each path component of $X\smallsetminus \bigcup C_\lambda$,  the edge and vertex groups are the fundamental groups of the corresponding subspaces of $X$, and the inclusions of edge groups into vertex group are induced by the inclusions of subspaces.

All groups can be realized as fundamental groups of reasonable spaces, with the nature of the space reflecting the nature of the group. For example, every finitely presented group is the fundamental group of a $2$-dimensional simplicial complex with finitely many cells, and of a smooth, compact $4$-dimensional manifold.  Thus the preceding discussion linking decompositions of spaces to
decompositions of groups relates to all groups. We refer the reader to the article of Scott and Wall \cite{Scott1979} for a thorough treatment of this and related topological methods in group theory.

The history of van Kampen's Theorem is presented well by Gramain \cite{Gramain1992}.  We will not rehearse it here, but we   mention a few highlights of note. In 1931, Seifert \cite{Seifert1931}, as part of his doctoral studies, proved van Kampen's theorem in the case where $X$ is a simplicial complex and $A$ and $B$ are subcomplexes of $X$. (This is why one sees references to the \textit{Seifert--van Kampen Theorem}).  In 1933, van Kampen \cite{vanKampen1933a, vanKampen1933b} proved the result in the generality we have described. His initial motivation came from the desire to have a topological method for computing
 the fundamental group of the complement of an algebraic curve in $\mathbb{CP}^2$,  a problem that Zariski had suggested to him.
Van Kampen solved this problem in  \cite{vanKampen1933b}, and in a second article immediately following it in the journal \cite{vanKampen1933a},  he clarifies the general theorem underlying his proof and states what we today recognize as van Kampen's Theorem. In a third paper immediately following these \cite{vanKampen1933c}
he introduces van Kampen diagrams and proves {\em van Kampen's Lemma}.   This  third
paper was overlooked for decades,  but the diagrammatic method introduced there played a central role in the
development of small cancellation theory from the 1960s onwards, and van Kampen's Lemma became a key tool in the study
of isoperimetric inequalities for groups following the work of Gromov \cite{gromov87, gromov93}; see \cite{mrb2002}.

\subsection{Graphs of groups}\label{ss:g-og-g}
The discussion of van Kampen's Theorem returns us to a point that emerged from a different viewpoint in the HNN paper \cite{HNN1949},  where it was repeatedly emphasized that it can be useful to perform several HNN extensions (and amalgamated
free products) simultaneously.  In anticipation of our discussion of Bass-Serre theory in  \S\ref{Subsec:Bass-Serre-theory},
we recall a useful formalism for doing this.

Following Serre \cite{SerreOriginal},  a {\em{graph}} is defined to consist of two sets $V$ and $E$ (thought of as vertices and oriented edges) and two maps (thought of as taking endpoints) $o,  t : E\to V$,  together with a free involution of $E$, denoted by $e\mapsto \-{e}$, such that $o(\-{e})=t(e)$ for all $e\in E$.  In practice, one pictures the topological graph whose vertices are indexed by $V$ and whose 1-cells are indexed by the ``unoriented edges", i.e. ~pairs $|e|:=\{e, \-{e}\}$. We restrict attention to the case where this topological graph is connected.

A {\em graph of groups} $\mathbb{G}$ is an assignment of groups $G_v  (v\in V)$  and $G_e  (e\in E)$  and injective homomorphisms $i_e: G_e\to G_{t(e)}$, with $G_e=G_{\-{e}}$. Serre \cite{SerreOriginal} introduces the {\em fundamental group} $\pi_1(\mathbb{G})$  of a graph of groups.  
In this language, the diagram (\ref{amalgam}) represents an ``edge of groups" and the fundamental group is the amalgamated free product associated to (\ref{amalgam}). More generally, if the underlying topological graph is a tree, then $\pi_1(\mathbb{G})$  is the pushout of the diagram of groups that represents it.  One sees that the natural maps from the vertex groups $G_v$ to this pushout are injective by starting with the one-edge case (amalgamated free products) and arguing by induction,  taking a direct limit if the graph is
infinite, this last step being what \cite{HNN1949} called a {\em union in the usual ``group tower" sense}.

When the underlying graph contains loops,  the definition of $\pi_1(\mathbb{G})$ is more subtle. The first case to consider is
where the graph   has one vertex $v$ and one unoriented edge $|e|$. In this case the fundamental group is the HNN extension of $G_v\ast_{G_e}$ corresponding to the isomorphism of subgroups 
$i_{e}\circ i_{\-{e}}^{-1}: i_{\-{e}}(G_e) \to i_e(G_e)$.  The fundamental group of an arbitrary graph of groups can be calculated in two steps: first we choose a maximal tree $T$ in the underlying graph and  take the fundamental group $\pi_1(\mathbb{T})$  of the  tree of groups over this, then we take a multiple HNN extension  of $\pi_1(\mathbb{T})$ with stable letters and associated subgroups indexed by the set of unoriented edges $\{e,\-{e}\}$ in the complement of $T$. It is important to note that the second step in this procedure corresponds exactly to the construction described in Theorem II of \cite{HNN1949} where arbitrary index sets are allowed. In this sense, the construction of the fundamental group of an arbitrary graph of groups is already present in \cite{HNN1949}, albeit in a different language, and without the  geometric imagery that renders the connection to van Kampen's theorem so transparent.

\subsection{Splittings, large scale geometry,   and hierarchies} \label{Subsec:splittings-hierarchies}
 
In the light of the preceding discussion, the reader will see why one says that a  group $G$ \textit{splits} over a subgroup $H<G$ if it decomposes as an amalgamated free product $G=A\ast_H B$ or an HNN extension $G=B\ast_H$.  (In \S\ref{Subsec:Bass-Serre-theory} we shall see that this is equivalent to saying that   $G$ acts on a simplicial tree and $H$ is the stabilizer of some edge in that tree.)
 
 In 1971, Stallings \cite{Stallings1971}  proved that a finitely generated group splits as an amalgamated free product or HNN extension over a finite group if and only if the Cayley graph of the group, with respect to any finite generating set, can be separated into two or more unbounded components by deleting a compact set.\footnote{Many of the key ideas behind this result can already be found in Stallings' 1968 paper covering the torsion-free case \cite{Stallings1968}; cf.\ also \cite{Cohen1970}.} This ``Ends Theorem" is one of the founding results of geometric group theory: it was revolutionary in establishing an equivalence between a feature of the large scale (``coarse") geometry of the group and a fine algebraic property of the group.  It inspired many related results of the following form: one proves that if a set that is small, in a suitable sense,  coarsely separates the Cayley graph of a group, then the group splits over a subgroup that is small in a corresponding sense.  Thus, for example,  Papasoglu \cite{Papasoglu2005} proved that if a finitely presented group is torsion-free and not a free product, then it splits over $\Z$ if and only if the Cayley graph can be coarsely separated in an essential way by a quasi-line (i.e.\ a connected subset that, in the induced metric, is Lipschitz equivalent to $\mathbb{R}$).   Bowditch \cite{Bowditch1998} proved another important result in this direction.

The relationship between the geometry of a finitely generated group,  the algebraic properties of the group,  
 and the ways in which the group splits is a central theme in modern group theory, too large to cover here,  but there are some highlights that we should mention because they are
important milestones in the evolution of HNN extensions.    \textit{One-relator groups} deserve particular mention.
The theory of such groups rests on inductive arguments that exploit
the {\em Magnus--Moldavanskii hierarchy,}  whereby a one-relator group is decomposed as
a one-relator group as an HNN extension of a ``smaller'' one-relator group with free associated subgroups; we refer the reader to \cite{McCool1973} for a classical perspective and more details, and \cite{LintonNB, Linton2025} for a more modern and topological viewpoint.  Much of the theory of {\em limit groups},  developed by Sela in connection with the Tarski problem (see \cite{sela} et seq.),  also rests on the existence of hierarchical decompositions involving amalgamated free products and HNN extensions of a controlled
form; see also \cite{KharlMias}. A wider class of groups intensively studied in contemporary geometric group theory  consists of the groups admitting \textit{quasiconvex hierarchies},  where the types of subgroup along which one splits is broader, but one retains tight control,
allowing one to build illuminating actions of the groups on non-positively curved spaces.  This rich and highly influential theory, developed by Dani Wise \cite{Wise2009},  has underpinned many of the most spectacular breakthroughs in geometric group theory
in recent years.

\subsection{Existence and uniqueness of splittings} Let us deal first with the issue of uniqueness.  
If a group splits as an  HNN extension, amalgamated free product or more generally  a graph of groups,
then it is natural to ask how unique this splitting might be, 
how one might classify and parameterize all splittings of a given type,  and under what circumstances  one say that 
there exist a canonical ``best'' splitting.  The study of such questions is a subtle subject that we cannot do justice to here,
but we will say something about \textit{JSJ decompositions}.

\textit{JSJ decompositions} originate in 3-manifold theory and the initials JSJ record the fundamental work of Jaco \& Shalen \cite{JacoShalen1978} and Johannson \cite{Johan1979}.  Every closed, orientable
3-manifold that is orientable can be cut into geometric pieces
along a  canonical collection of  embedded,  $\pi_1$-injective,   annuli and tori.  According to
van Kampen's Theorem,  this decomposes the fundamental group of the manifold as a graph of groups with edge groups
that are cyclic or free-abelian of rank $2$. Remarkably,  key features of this decomposition, which {\em a priori} seems so rooted in 3-manifold topology,  can be generalized to finitely presented groups, provided that one restricts attention to splittings with
edge groups that are small in a suitable sense, e.g.\ cyclic or abelian.  
 This revelation was established in the work of Rips \& Sela \cite{Rips1997} and advanced by many authors, notably Bowditch \cite{Bowditch1998}, Dunwoody \& Sageev \cite{DS99},  and Fujiwara \& Papasoglu \cite{FP06}. The extent to which the splitting is unique depends heavily on the class of edge groups that one allows
in the splitting.  We refer to the monograph of Guirardel \& Levitt \cite{GuirLev} for further references and a thorough account of the state of the art, including the theory of deformation spaces,  which provides the language  needed to express uniqueness appropriately.
 
 Let us turn now to the question of existence of splittings.  Given a group,  how can one decide
 whether or not it splits as an HNN extension or amalgamated free product in a non-trivial way? This is a very difficult question, indeed undecidable in general, and to give an indication of the difficulty we note some natural, and historically pivotal, examples coming from matrix groups. It is well-known and classical that $\SL_2(\Z)$ can be expressed
 as an  amalgamated free product of cyclic groups $C_4 \ast_{C_2} C_6$.  Less obviously, 
 $\SL_2(\Q_p)$ also  splits  --  in a ``somewhat mysterious manner'',  to quote J-P. Serre \cite[p.~3]{SerreOriginal}.  
  In contrast,  for $n \geq 3$ the group $\SL_n(\Z)$ does \textit{not} split as a  non-trivial  HNN extension or amalgamated free product.  One cannot hope to see this by staring at group presentations, so how is one to prove it?  This is the starting
  point of {\em Bass-Serre theory} -- instead of staring at a presentation of $\SL_3(\Z)$ and trying to figure out why
  it does {\em not} do something,  one translates the question into one about {\em group actions} and proves that there is
  an equivalent thing that the group {\em does} do -- in this case, fix a point whenever it acts on a tree.

\subsection{Bass-Serre theory: trees, amalgams, and $\operatorname{SL}_2$}\label{Subsec:Bass-Serre-theory}
The Nielsen--Schreier Theorem states that \textit{every subgroup of a free group is free}. Nielsen \cite{Nielsen1921} gave a proof of this for finitely generated subgroups in 1921, with a purely combinatorial argument, whereas Schreier \cite{Schreier1927} gave a more topological proof in 1927.  To a topologist,  thinking of free groups as the fundamental groups of graphs,  this theorem is an immediate
consequence of the theory of covering spaces:  \textit{a group is free if and only if it acts freely on a tree}.  
Serre \cite[Théorème~4]{SerreOriginal}  gives this proof 
and it serves as a paradigm for his theory of groups acting on trees,  wherein he provides 
elegant {\em conceptual} proofs for classical results about HNN extensions and amalgamated free products, and opens the door to a host of powerful applications. 

Important motivation for Serre came from the appearance of trees in the work of Jacques Tits on algebraic and simple groups \cite{Tits1964,  Tits1970}.  Another important string of ideas can be traced through the 1959 paper of H.\ Nagao \cite{Nagao1959} where he proved that $\GL_2(\mathbf{k}[x])$ is not finitely generated whenever $\mathrm{k}$ is a field: he proved that $\GL_2(\mathbf{k}[x])$ admits a decomposition as a free product with amalgamation (although without using this language), and his proof of this fact is essentially a direct sequence of manipulations of relations combined with elementary results about matrices.  In line with Nagao's result,  Y.\ Ihara \cite{Ihara1966b}
whose motivations were number-theoretic,  proved  the following famous theorem by expressing $\SL_2(\Q_p)$ as the free product of two copies of $\SL_2(\Z_p)$ over a common subgroup:
{\em every discrete torsion-free subgroup of $\SL_2(\Q_p)$ is free.} Demystifying this result was the aim of a course Serre gave in 1968/69 at the Coll\`{e}ge de France; a course that would lay out what is today known as \textit{Bass--Serre theory}. 

Serre's treatment of the subject is based on his definition of a {\em graph of groups}, which we described in \S\ref{ss:g-og-g}. The real power of Serre's work lies with his interpretation of a graph of groups $\mathbb{G}$ as the quotient data for an action of $\G=\pi_1\mathbb{G}$ on a tree and, crucially, his proof that one can reconstruct $\G$ and the action on the tree from this data. Specifically, if $\G$ is a group acting (without inversions) on a tree $X$, then we can associate a
 {\em quotient  graph of groups} $\mathbb{G}(\G, X)$ to this action: its underlying graph is the quotient $X / \G$, and the vertex and edge groups of $\mathbb{G}(\G, X)$ are 
 stabilizers of vertices and edges in a fundamental domain for the action (with care needed when the fundamental domain is not strict). The following result  is sometimes called the ``fundamental theorem of Bass--Serre theory''.

\begin{theorem*}[{Serre \cite[Théorème~13]{SerreOriginal}}] If $\G$ is a group acting without inversions on a tree $X$,
then $\G\cong \pi_1 \mathbb{G}(\G, X)$.  Conversely,  for every graph of groups $\mathbb{G}$, there exists an action of $\pi_1\mathbb{G}$ on a tree  $X$ such that the quotient graph of groups is $\mathbb{G}$.
\end{theorem*}
The tree $X$ in the second clause is called the {\em Bass-Serre tree} of the splitting $\mathbb{G}$.

We pause briefly to discuss the history of Serre's famous book and the way in which the material is presented. In his course at the Coll\`{e}ge de France, Serre only discussed the cases where the quotient graph of the action  is a single edge or a loop,  corresponding to the case of amalgamated free products and HNN extensions. In the audience was H.\ Bass, who in the fall of 1968 was on sabbatical in Paris, and who had collaborated with Serre (and Milnor) on the solution to the Congruence Subgroup Problem \cite{Bass1967}.  Serre invited Bass to develop notes for the course, which he gladly accepted. The two only had brief, infrequent consultations in Parisian sidewalk cafés regarding the notes,  but at one of these conversations Serre suggested that the two cases detailed in the course (segments and loops) might be developed into a more general theory. Bass took on this task and developed the notion of a graph of groups as we have presented it, giving rise to Bass--Serre theory as we know it today. The notes appeared with the title ``\textit{Arbres, amalgames, $\operatorname{SL}_2$}'' \cite{SerreOriginal},  later translated by J.\ Stillwell into English with the simpler title ``\textit{Trees}'' \cite{SerreTranslated}.\footnote{The notes were also privately circulated earlier under the title ``\textit{Groupes discrets}'', which consisted approximately of Chapter I and the first part of Chapter II of the later published notes.} A curious point is that the terminology ``HNN extensions'' and the associated results from \cite{HNN1949} seem to have been absent from the course. As Bass gave seminars on Bass--Serre theory at Queen Mary College in London and elsewhere,  the link with the established usage was made and found its way, via Bass, into the  notes \cite{SerreOriginal}.

We now return to the original applications of Bass--Serre theory. There are two parts of the notes \cite{SerreOriginal} -- Chapter I and Chapter II. The second chapter is devoted to the study of $\SL_2$ and  $\GL_2$ over various rings and fields, via their actions on trees.  It includes elegant conceptual proofs of the theorems of Nagao and Ihara mentioned above as well as  a number of applications of a more arithmetic nature.  But since our main interest is in following the development of HNN extensions, we concentrate on the first chapter,  which has become required reading for all beginning students of geometric group theory.

Two results presented by Serre serve to underscore clearly the conceptual advantage of shifting from an algebraic viewpoint on HNN extensions and amalgamated free products to the viewpoint of groups acting on trees; both can be proved algebraically and were well known at the time, but in both cases the geometric viewpoint lends new insight. First, there is the theorem that a finite subgroup
of an HNN extension $G\ast_H$ must be conjugate to a subgroup of $G$, and similarly any finite subgroup of an amalgamated free product $A\ast_C B$ must be conjugate to a subgroup of $A$ or $B$. In both cases, this is an immediate consequence of the easy fact that a finite group acting on a tree without inversions must fix a vertex.

The second result, which we state  only for the HNN case $\G = G\ast_H$, is the following subgroup result: if a subgroup $\Lambda < \G$ does not intersect any conjugate of $H$ non-trivially, then it is a free product of a free group and groups of the form $\Lambda\cap G^x$, where $G^x$ is a conjugate of $G$ in $\Lambda$.  This is proved by noting that $\Lambda$ acts on the  Bass-Serre tree $X$ for the splitting $\G = G\ast_H$ and the edge-stabilizers are the intersections of $\Lambda$ with the conjugates of $H$,  hence are trivial,  while  the vertex stabilizers are intersections of $\Lambda$ with conjugates of $G$.  The splitting  of $\Lambda$ described in the theorem is the quotient $\mathbb{G}(X,\Lambda)$. One can sharpen the theorem by noting that the vertices in the quotient graph  are indexed by double-cosets $\Lambda\backslash \G /G$. As a special case, this theorem recovers the classical Kurosh Subgroup Theorem for free products of groups. 

These applications provide elegant insights into previously known results, but the next application that we want to discuss
is much more original and powerful. The methods of Bass--Serre theory provide an indispensable tool for showing that certain groups do \textit{not} decompose as an HNN extension or  amalgamated free products. 
Following Serre,  one says that a group $G$ has \textit{property} (\rm{FA})\footnote{Of course, the letter F indicates (point) \textit{fixe}, and the letter A indicates \textit{arbre}.} if   $G$
fixes a point whenever it acts on a tree with inversions.  If a group has this property then it cannot split non-trivially as an HNN extension or amalgamated free product.  Serre phrases this  as follows: for countable groups
$G$, property $\mathrm{(FA)}$ is equivalent to the conjunction of the following three conditions,
\begin{enumerate}
\item $G$ does not split as a non-trivial amalgamated free product,
\item $H^1(G, \Z) = 0$, i.e.\ $G$ does not surject $\Z$. 
\item $G$ is finitely generated.  
\end{enumerate}
The case of splitting as an HNN extension is absorbed into the second condition. The finite generation condition is necessary because if $G$ can be written as a strictly ascending union $G=\bigcup_n G_n$, then $G$ is the fundamental group of the graph of groups with vertices indexed by $n$, where the vertex group at $v_n$ is $G_n$ and the  group on the edge joining   $v_n$ to $v_{n+1}$ is also $G_n$; the action on the Bass-Serre tree of this splitting does not have a fixed point.  (This last argument points to the fact that one can drop
the assumption that $G$ is countable if one replaces the  condition of finite generation with the assumption that $G$ is not a strictly ascending union of a sequence of subgroups.)

One can use this theorem to prove that  $\SL_n(\Z)$ does not split  non-trivially if $n\ge 3$ by showing
that it fixes a point whenever it acts on a tree.  For  $\SL_n(\Z)$ this is not difficult to prove --  in Serre's notes it takes less than a page \cite[p. 94]{SerreOriginal} -- and many other groups of geometric interest have since been shown to have property (FA).

\subsection{Developments from Bass-Serre Theory}\label{Subsec:modern-Bass-Serre} 
We have seen how the theory of HNN extensions got absorbed into Bass-Serre theory.   
In the early 1970s, there were still many results appearing in the literature concerning the subgroup structure of free products and HNN extensions,  expressed in the traditional language of combinatorial group theory, e.g.~\cite{Karrass1970, Karrass1971, Burns1972, Burns1973}. 
D.\ E.\ Cohen, who learned of Bass--Serre theory from the privately circulated notes predating \cite{SerreOriginal}, realized that the techniques of this new theory could be used to give conceptual proofs of these subgroup theorems, and indeed could strengthen some of them, essentially by arguing with the covering-space theory of graphs-of-groups. Cohen explained this at the conference \textit{Infinite Groups} in Calgary in 1974 and in the subsequent papers \cite{Cohen1974, Cohen1976}.  

 I.\ M.\ Chiswell \cite{Chiswell1976b} was another early adopter of Bass-Serre theory,  using
it to give a new proof of the \textit{Grushko--Neumann Theorem} (proved independently by I.\ A.\ Grushko \cite{Grushko1940} and B.\ H.\ Neumann \cite{Neumann1943}), which states that the minimum number of elements needed to generate a group is additive with respect to free products: $d(G_1 \ast G_2) = d(G_1) + d(G_2)$.  In an important demonstration of the power of topological methods, J.R.~Stallings \cite{stall1965} had given an elegant and conceptual proof of this theorem in 1965, greatly simplifying the original,  and Chiswell's proof refined this further.  In a different direction, Chiswell's work on length functions for groups,  which built on an idea of
 Lyndon \cite{Lyndon1963},  also played an important role, along with the paper of Alperin and Moss \cite{AM85}, in the broadening of Bass-Serre theory towards actions on $\mathbb{R}$-trees. By definition, an $\mathbb{R}$-tree is a geodesic metric space in which every pair of points is joined by a unique topological arc. This subject was developed extensively by M.~Culler,  J.~Morgan and P.B.~Shalen \cite{CM,CS83,MS84} with an eye to important applications in low-dimensional topology. It was realized that group actions on $\mathbb{R}$-trees appear naturally as points at infinity in character varieties and as dual objects parameterizing the measured laminations that appear as ideal points in Thurston's compactification of Teichm\"{u}ller space \cite{MS84}. The theory of groups acting on  $\mathbb{R}$-trees blossomed into a rich subject with myriad applications across geometric group theory and low dimensional topology; we refer to \cite{Best,Shalen} for an overview. A crucial link between group actions on $\mathbb{R}$-trees and classical Bass-Serre theory was forged by E.~Rips who developed a classification theory for group actions on $\mathbb{R}$-trees that, under reasonable hypotheses,  allows  one to promote actions of finitely presented groups on $\mathbb{R}$-trees that have small  arc stabilizers,  to actions on simplicial trees with edge-stabilizers that
 are similarly small \cite{BesFei, GLP}. He also proved that the only finitely generated groups that can act freely on $\mathbb{R}$-trees are free products of free and surface groups.

We should also mention the important article of Scott \& Wall \cite{Scott1979} from 1979, which was very influential in conveying the ideas and potential of Bass-Serre theory to a broad audience,  recasting it in the familiar language of covering space theory, and  framing many of its ramifications in the context of the wider development of  topological methods in group theory, which was the subject of a landmark LMS  Symposium in Durham in 1977.

It would be remiss to end this tour through the legacy of the HNN paper without touching on the issue of whether the Bass-Serre theory of graphs-of-groups can be extended to a theory of higher-dimensional {\em complexes of groups}, and an associated theory
of groups acting on higher-dimensional analogues of trees. In short,  there is such a theory, indeed a rich and powerful one, developed  by A.~Haefliger \cite{haef} and laid out most fully in \cite[Chapter III.C]{BH}.  But the Developability Theorem at the heart of the theory
is by necessity less straightforward than the fundamental theorem of Bass-Serre theory,  and the general theory is
a lot more technical than in dimension one,  with correspondingly fewer applications.  A key feature of 
the theory in higher dimensions is that it works best in the presence of a non-positive curvature condition on the underlying spaces, a feature that is obscured in the one-dimensional case (graphs) by the fact that in the obvious metric, all graphs are negatively curved. 

We close by emphasizing once again that HNN extensions, used explicitly or embedded in Bass-Serre theory, are indispensable tools at the heart of modern geometric group theory,  with far more applications than can touched upon in an article of this length. The memorable applications given in the 1949 paper \cite{HNN1949} -- embedding countable groups into 2-generator groups, and constructing infinite groups with only two conjugacy classes -- remain landmark theorems,  but the abiding importance of \cite{HNN1949} lies with the authors' discovery of the HNN extension,  which they used to solve the most pressing question of their time: {\em A group $G$ contains two subgroups $A$ and $B$. The question then arises whether a group $H$ exists, containing $G$, within which $A$ and $B$ are conjugate.}

%\begin{multicols}{0}
\bibliographystyle{amsalpha}
\bibliography{CF-MB-HNN-submit.bib}
%\end{multicols}

\end{document}